\documentclass[submission,copyright,creativecommons]{eptcs}
\providecommand{\event}{Non-Classical Logics 2024 (NCL'24)}

\usepackage{iftex}

\ifpdf
  \usepackage{underscore}         % Only needed if you use pdflatex.
  \usepackage[T1]{fontenc}        % Recommended with pdflatex
\else
  \usepackage{breakurl}           % Not needed if you use pdflatex only.
\fi

\usepackage{graphicx}

\usepackage{amsmath}
\usepackage{amssymb}
\usepackage{cases}
\usepackage{comment}
\usepackage[noadjust]{cite}
\usepackage{enumerate}
\usepackage{mathrsfs}
\usepackage{mathtools}
\usepackage{proof}
\usepackage{qtree}
\usepackage{subcaption}
\usepackage{siunitx}
\usepackage{url}
\usepackage{xcolor}
\definecolor{1f1e33}{HTML}{1F1E33}
\definecolor{mediumblue}{HTML}{0000CD}

\usepackage{tikz}
\usetikzlibrary{automata}

\usepackage[all]{xy}

\usepackage{amsthm}

\theoremstyle{definition}
\newtheorem{dfn}{Definition}[section]
\newtheorem{ex}[dfn]{Example}
\newtheorem{prop}[dfn]{Proposition}
\newtheorem{lem}[dfn]{Lemma}
\newtheorem{thm}[dfn]{Theorem}
\newtheorem{cor}[dfn]{Corollary}

\usepackage{listings}
\newcommand{\Prop}{\mathbf{Prop}}
\newcommand{\Nom}{\mathbf{Nom}}

\newcommand{\pset}[1]{\mathcal{P}( {#1} )}

\newcommand{\Dia}{\Diamond}
\newcommand{\Kmodel}{\mathcal{M}}
\newcommand{\agent}{\mathbf{A}}
\newcommand{\langEL}{\mathcal{L}_\text{EL}}
\newcommand{\langAK}{\mathcal{L}_\text{AK}}
\newcommand{\MEL}{\mathcal{M}_\text{EL}}
\newcommand{\MAK}{\mathcal{M}_\text{AK}}
\newcommand{\modelsEL}{\models_\text{EL}}
\newcommand{\modelsAK}{\models_\text{AK}}
\newcommand{\tableauAK}{\mathbf{T}_\text{AK}}
\newcommand{\reflexivity}{\mathit{Ref}}
\newcommand{\Id}{\mathit{Id}}

\title{Agent-Knowledge Logic for Alternative Epistemic Logic}
\author{Yuki Nishimura\\
\institute{Tokyo Institute of Technology}\\\email{nishimura.y.as@m.titech.ac.jp}}
\def\titlerunning{Agent-Knowledge Logic for Alternative Epistemic Logic}
\def\authorrunning{Y. Nishimura}

\begin{document}
\maketitle

\begin{abstract}
    Epistemic logic is known as a logic that captures the knowledge and beliefs of agents and has undergone various developments since Hintikka (1962). In this paper, we propose a new logic called agent-knowledge logic by taking the product of individual knowledge structures and the set of relationships among agents. This logic is based on the Facebook logic proposed by Seligman et al. (2011) and the Logic of Hide and Seek Game proposed by Li et al. (2021). We show two main results; one is that this logic can embed the standard epistemic logic, and the other is that there is a proof system of tableau calculus that works in finite time. We also discuss various sentences and inferences that this logic can express.
\end{abstract}

\section{Introduction}
\label{secintro}

Investigations into knowledge and beliefs form part of philosophy, which is now called epistemology. This area has been the subject of various studies from the standpoint of logic. One of these was conducted by applying modal logic, which is nowadays called epistemic logic. The operator $K_i$, which is the key element of this logic, has form $K_i \varphi$, which expresses that ``agent $i$ knows that $\varphi$.'' On this basis, it is possible to represent various concepts related to knowledge and belief in formal language. As far as I know, the pioneering work on epistemic logic was done by Hintikka in 1962 \cite{Hintikka1962}, and there is a wide range of research being done today; see Fagin et al. \cite{fagin1995} and van Benthem \cite{vanbenthem2006}.

A more recent logic for human knowledge is Facebook logic, developed by Seligman et al. in 2011 \cite{seligman2011}. This logic was invented to describe personal knowledge plus the friendships of agents in two-dimensional hybrid logic. For instance, consider this sentence: ``I am Andy's friend, and Andy knows he has pollen allergy. Then, one of my friends knows that they have pollen allergy.'' This inference can be written using the language of Facebook logic as follows:
\[
    \langle \texttt{Friend} \rangle i \land @_i [\texttt{Know}] p \rightarrow \langle \texttt{Friend} \rangle [\texttt{Know}] p
\]
where $p =$ ``they have pollen allergy.'' and $i =$ ``This is Andy.''
The at sign $@$ in the logical formula is the operator of hybrid logic, where $@_i p$ can be read as ``$p$ holds at point $i$.'' Facebook logic uses nominals, a tool of hybrid logic, to make reference to individual agents. For a thorough introduction into hybrid logics, we refer the reader to Blackburn \& ten Cate \cite{blackburn2006P}, Indrzejczak \cite{indrzejczak2007}, and Bra\"{u}ner \cite{brauner2011}. Sano \cite{sano2010} provides further details on two-dimensional hybrid logic.

In fact, Facebook logic treats propositional variables differently from epistemic logic. The truth of a propositional variable $p$ depends not only on the epistemic alternative but also on the agent under consideration. Therefore, the propositions represented by the propositional variables here are personal properties, such as, ``I have a pollen allergy.''

The new logic proposed in this paper --- we will call it \emph{agent-knowledge logic} --- is a modification of the aforementioned Facebook logic. One feature of this logic is that the fragment of it is compatible with epistemic logic. This property allows us to use agent-knowledge logic as an alternative to epistemic logic. Indeed, this paper shows how to embed epistemic logic into our new logic. Furthermore, agent-knowledge logic is able to formalize a variety of sentences that cannot be represented by traditional epistemic logic, such as ``one of my friends knows $p$.'' Some of the examples given in this paper may be only part of the possibilities our new logic opens up.

In this paper, we also introduce a proof system, by constructing a tableau calculus. The tableau calculus is not only a proof system but also a system for discovering a counterexample model in which the formula is not valid. In particular, by constructing a tableau calculus with the termination property --- in short, that the proof ends in finite time --- we can show that the logic is decidable.

This logic has two parents: one is Facebook logic, and the other is, which seems to have nothing to do with epistemic logic, the Logic of Hide and Seek Game (LHS, in short) created by Li et al. in 2021 \cite{li2021, li2023}. This logic was originally invented to illustrate the hide and seek game (also known as cops and robbers). In LHS, propositional variables are split into two sets, which are related to hider and seeker, respectively. We borrow this idea to express the \emph{agent-free} propositions (``the sun rises in the east,'' for example.)

We proceed as follows: Section \ref{secpre} reviews the well-known epistemic logic and explains the parents of agent-knowledge logic, Facebook logic, and LHS, briefly. In Section \ref{secAK}, we introduce our new logic, that is, agent-knowledge logic. Section \ref{secembed} shows how we embed epistemic logic into our new logic. In Section \ref{sectableau}, we construct a tableau calculus with the termination property and completeness. Finally, in Section \ref{secfuture}, we write about some future prospects.

\section{Preliminary}
\label{secpre}
\subsection{Epistemic Logic}
\label{subsecEL}
This section is mostly based on the work of Fagin et al. \cite[Chapter 2]{fagin1995}.

In epistemic logic, we have another set $\agent$ of agents besides a usual set $\Prop$ of propositional variables. The elements of $\agent$ occur in a new operator $K_i$. The intuitive meaning of $K_i \varphi$ is that ``agent $i$ knows $\varphi$.''

\begin{dfn}
    We have two disjoint sets, $\Prop$ and $\agent$. A formula $\varphi$ of the \emph{epistemic logic} $\langEL$ is defined as follows:
    \[
        \varphi \Coloneqq p \mid \neg \varphi \mid \varphi \land \varphi \mid K_i \varphi
    \]
    where $p \in \Prop$ and $i \in \agent$.
    \label{deflangEL}
\end{dfn}

We only use $\neg$ and $\land$ as primitives since other Boolean operators, such as $\lor$ and $\rightarrow$, can be defined as compounds of the first two operators.

\begin{dfn}
    A \emph{Kripke model for epistemic logic} (we call it \emph{EL model}) $\MEL$ is a tuple $(W, (R_i)_{i \in \agent}, V)$ where
    \begin{itemize}
        \item $W$ is a non-empty set,
        \item For each $i \in \agent$, $R_i$ is a binary relation on $W$,
        \item $V: \Prop \rightarrow \pset{W}$.
    \end{itemize}
\end{dfn}

\begin{dfn}
    Given an EL model $\MEL$, its point $w$, and a formula $\varphi \in \langEL$, the \emph{satisfaction relation} $\MEL, w \models \varphi$ is defined inductively as follows:
    \begin{align*}
        \MEL, w \models p \ &\iff \ w \in V(p) \text{ where } p \in \Prop, \\
        \MEL, w \models \neg \varphi \ &\iff \ \text{Not } \MEL, w \models \varphi \ (\MEL, w \not \models \varphi), \\
        \MEL, w \models \varphi \land \psi \ &\iff \ \MEL, w \models \varphi \text{ and } \MEL, w \models \psi, \\
        \MEL, w \models K_i \varphi \ &\iff \ \text{For all } v \in W, w R_i v \text{ implies } \MEL, v \models \varphi.
    \end{align*}
\end{dfn}

As for epistemic logic, we define the validity of a formula. Later we discuss embedding epistemic logic into our new logic, so the formal definition is needed.

\begin{dfn}
    A formula $\varphi$ is \emph{valid} with respect to the class of EL models (written as $\modelsEL \varphi$) if $\MEL, w \models \varphi$ for every model $\MEL$ and its every world $w$.
\end{dfn}

\subsection{Facebook Logic}
\label{subsecFL}
Facebook logic, firstly invented by Seligman et al. \cite{seligman2011}, has two characteristics compared to classical modal logic.

First, we have two modal operators, $K$ and $F$. These modal operators correspond to knowledge and friendship, respectively. Correspondingly, a possible world is decomposed into two components: one representing an agent and the other representing an epistemic alternative of an individual.

Another addition is the introduction of special propositional variables called \emph{nominals}. A nominal $n$ is a proposition corresponding to only one agent, which is a proposition for the \emph{name} of the agent. In addition, we introduce the satisfaction operator $@$ used in hybrid logic. The intuitive meaning of $@_n p$ is that ``$p$ holds for agent $n$.''

% Formally, nominal is defined as ``a proposition such that the image by $V$ is $W \times \{ a \}$.'' In other words, one

Let us introduce a formal definition. We have two disjoint infinite sets, $\Prop$ of propositional variables and $\Nom$ of nominals. A formula $\varphi$ of the \emph{Facebook logic} is defined as follows:
\[
    \varphi \Coloneqq p \mid n \mid \neg \varphi \mid \varphi \land \varphi \mid K \varphi \mid F \varphi \mid @_n \varphi
\]
where $p \in \Prop$ and $n \in \Nom$. If needed, we can define the dual $\langle K \rangle$ and $\langle F \rangle$ of each modal operators as $\langle K \rangle \varphi \coloneqq \neg K \neg \varphi$ and $\langle F \rangle \varphi \coloneqq \neg F \neg \varphi$.

The semantics of Facebook logic is based on \emph{epistemic social network models}. An epistemic social network model is a tuple $(W, A, (\sim_a)_{a \in A}, (\asymp_w)_{w \in W}, V)$, where
\begin{itemize}
    \item $W$ is a set of epistemic alternatives,
    \item $A$ is a set of agents,
    \item For each $a \in A$, $\sim_a$ is an equivalence relation on $W$,
    \item For each $w \in W$, $\asymp_w$ is an irreflexive and symmetric relation of friendship on $A$, and
    \item $V$ is a valuation function, which assigns a propositional variable $p$ to a subset of $W \times A$ and a nominal $n$ to a set $W \times \{ a \}$ for some $a \in A$.
\end{itemize}
The reason for a relation $\asymp_w$ over $A$ being irreflexive and symmetric can be understood when we assume it as a friendship; no one is a friend to oneself, and if a person is your friend, then you are a friend of them.

Then, the truth of formulas in Facebook logic is defined inductively. The Boolean cases are omitted since they are the same as those in classical modal logic. Also, the element $a \in A$ such that $V(n) = W \times \{ a \}$ holds is abbreviated as $n^V$.
\begin{align*}
    \Kmodel, w, a \models p \ &\iff \ (w, a) \in V(p) \text{ where } p \in \Prop, \\
    \Kmodel, w, a \models n \ &\iff \ n^V = a, \text{ where } n \in \Nom\\
    \Kmodel, w, a \models K \varphi \ &\iff \ \Kmodel, v, a \models \varphi \text{ for every } v \sim_a w, \\
    \Kmodel, w, a \models F \varphi \ &\iff \ \Kmodel, w, b \models \varphi \text{ for every } b \asymp_w a, \\
    \Kmodel, w, a \models @_n \varphi \ &\iff \ \Kmodel, w, n^V \models \varphi. \\
\end{align*}
As mentioned in the Introduction, the truth of a propositional variable depends on both an epistemic alternative and an agent.

\begin{ex}
    The following formulas of Facebook logic can be translated into natural language as follows.
    \begin{itemize}
        \item $K p$: I know that I am $p$.
        \item $K F p$: I know that all of my friends are $p$.
        \item $F K p$: Each of my friends knows that they are $p$.
        \item $\langle F \rangle n$: I have a friend $n$.
        \item $@_n K p$: An agent $n$ knows that they are $p$.
    \end{itemize}
\end{ex}

For the readers who would like to study it deeper, Seligman et al. \cite{seligman2011} and its sequel, Seligman et al. \cite{seligman2013}, should be of help.

\subsection{Logic of Hide and Seek Game}
\label{subsecLHS}

The logic of hide and seek game (LHS), as the name implies, is a logic for describing a hide and seek game. There are two players, a hider and a seeker, and a set of propositional variables $\Prop_H$ and $\Prop_S$ for each player to describe their state. Moreover, there is a special propositional variable $I$. This is a proposition to describe that the hider and seeker are in the same place, i.e., expressing ``I find you!''

The main difference from Facebook logic is that we use the same structure $(W, R, V)$ as in usual modal logic, which is appropriate considering that the hide and seek game is played by two players on the same board.

Here is a definition of a formula of LHS $\varphi$, where $p_H \in \Prop_H$ and $p_S \in \Prop_S$.
\[
    \varphi \Coloneqq p_H \mid p_S \mid I \mid \neg \varphi \mid \varphi \land \varphi \mid \Dia_H \varphi \mid \Dia_S \varphi
\]

The truth value of LHS formulas is defined inductively as follows. Note that both $x$ and $y$ are elements of $W$.
\begin{align*}
    \Kmodel, x, y \models p_H \ &\iff \ x \in V(p_H) \text{ where } p_H \in \Prop_H, \\
    \Kmodel, x, y \models p_S \ &\iff \ y \in V(p_S) \text{ where } p_S \in \Prop_S, \\
    \Kmodel, x, y \models I \ &\iff \ x = y, \\
    \Kmodel, x, y \models \Dia_H \varphi \ &\iff \ \text{there is some } x' \text{ such that } x R x' \text{ and } \Kmodel, x', y \models \varphi, \\
    \Kmodel, x, y \models \Dia_S \varphi \ &\iff \ \text{there is some } y' \text{ such that } y R y' \text{ and } \Kmodel, x, y' \models \varphi. \\
\end{align*}

Using this language, we can describe the hide and seek game. For example, $\Box_H \Dia_S I$ means that no matter how the hider moves, the seeker has a one-step move to catch the hider. This expression shows the existence of a winning strategy for the seeker.

In addition to the already mentioned Li et al. \cite{li2021}, Li et al. \cite{li2023} may also help readers who want to know more about LHS.

\section{Agent-Knowledge Logic}
\label{secAK}
Here, we introduce a new logic, called \emph{agent-knowledge logic}. As you read in Section \ref{secintro}, this logic is a mixture of Facebook logic and LHS. We have two dimensions, which correspond to agents and their knowledge, respectively. This structure and the intention behind it are very similar to that of Facebook logic. On the other hand, the idea that we use both $\Prop_A$ and $\Prop_K$ is unique for LHS.

\subsection{Agent-Knowledge Model}

\begin{dfn}
    We have four disjoint sets $\Prop_A$, $\Prop_K$, $\Nom_A$, and $\Nom_K$. A formula $\varphi$ of the \emph{agent-knowledge logic} $\langAK$ is defined as follows:
    \[
        \varphi \Coloneqq p_A \mid p_K \mid a \mid k \mid \neg \varphi \mid \varphi \land \varphi \mid \Box_A \varphi \mid \Box_K \varphi \mid @_a \varphi \mid @_k \varphi
    \]
    where $p_A \in \Prop_A$, $p_K \in \Prop_K$, $a \in \Nom_A$, and $k \in \Nom_K$.
\end{dfn}

We call an element of both $\Nom_A$ or $\Nom_K$ a \emph{nominal}. As we mentioned in Section \ref{subsecFL}, they point to a specific agent and a specific epistemic alternative, respectively. As well as $\lor$ and $\rightarrow$, if we need, we can define $\Dia_A$ and $\Dia_K$ in the usual way.

\begin{dfn}
    A \emph{agent-knowledge model (AK model)} $\MAK$ is a tuple \\ $(W_A, W_K, (R_y)_{y \in W_K}, (S_x)_{x \in W_A}, V_A, V_K)$ where
    \begin{itemize}
        \item $W_A, W_K$ are non-empty sets,
        \item For each $y \in W_K$, $R_y$ is a binary relation on $W_A$,
        \item For each $x \in W_A$, $S_x$ is a binary relation on $W_K$,
        \item $V_A: \Prop_A \cup \Nom_A \rightarrow \pset{W_A}$ where if $a \in \Nom_A$, then $V_A(a) = \{ x \} \text{ for some } x \in W_A$,
        \item $V_K: \Prop_K \cup \Nom_K \rightarrow \pset{W_K}$ where if $k \in \Nom_K$, then $V_K(k) = \{ y \} \text{ for some } y \in W_K$.
    \end{itemize}
\end{dfn}

Note that the image of a nominal $\Nom_A$ by $V_A$ is a singleton (the same fact holds for $\Nom_K$ and $V_K$.) Owing to this definition, nominal behaves as a \emph{name} for each possible world.

We can illustrate an agent-knowledge model as if we write Cartesian coordinates in Figure \ref{figAKmodel}. In this circumstance, a nominal is represented as a horizontal or vertical line. Likely, a propositional variable is depicted as a set of parallel lines.

We write $V$ to express $V_A \cup V_K$. For instance, $V(p_A) = V_A(p_A)$. Moreover, we abbreviate $x \in W_A$ such that $V_A(a) = \{ x \}$ by $a^V$. We do the same for $k^V$.

\begin{figure}[tb]
  \centering
  \begin{tikzpicture}
    \draw[->,>=stealth](-0.5,0)--(2.5,0)node[below]{$W_A$};
    \draw[->,>=stealth](0,-0.5)--(0,2.5)node[left]{$W_K$};
    \draw[very thick] (1.2,0)--(1.2,2.5);
    \draw (1.20,1)node[right]{$\{ a^V \} \times W_K$};
    \draw (1.2,-0.05)node[below]{$a^V$};
    \draw[->,>=stealth](4.5,0)--(7.5,0)node[below]{$W_A$};
    \draw[->,>=stealth](5,-0.5)--(5,2.5)node[left]{$W_K$};
    \draw[very thick] (5,0.7)--(7.5,0.7);
    \draw[very thick] (5,0.9)--(7.5,0.9);
    \draw[very thick] (5,1.3)--(7.5,1.3);
    \draw (6.2,1.30)node[above]{$W_A \times V(p_K)$};
  \end{tikzpicture}
  \caption{An agent-knowledge model.}
  \label{figAKmodel}
\end{figure}

\begin{dfn}
    Given a model $\MAK$, its points $(x, y) \in W_A \times W_K$, and a formula $\varphi \in \langAK$, the \emph{satisfaction relation} $\MAK, (x, y) \models \varphi$ is defined inductively as follows:
    \begin{align*}
        \MAK, (x, y) \models p_A \ &\iff \ x \in V(p_A) \text{ where } p_A \in \Prop_A, \\
        \MAK, (x, y) \models p_K \ &\iff \ y \in V(p_K) \text{ where } p_K \in \Prop_K, \\
        \MAK, (x, y) \models a \ &\iff \ x = a^V \text{ where } a \in \Nom_A, \\
        \MAK, (x, y) \models k \ &\iff \ y = k^V \text{ where } k \in \Nom_K, \\
        \MAK, (x, y) \models \neg \varphi \ &\iff \ \text{Not } \MAK, (x, y) \models \varphi \ (\MAK, (x, y) \not \models \varphi), \\
        \MAK, (x, y) \models \varphi \land \psi \ &\iff \MAK, (x, y) \models \varphi \text{ and } \MAK, (x, y) \models \psi, \\
        \MAK, (x, y) \models \Box_A \varphi \ &\iff \ \text{For all } x' \in W_A, x R_y x' \text{ implies } \MAK, (x', y) \models \varphi, \\
        \MAK, (x, y) \models \Box_K \varphi \ &\iff \ \text{For all } y' \in W_K, y S_x y' \text{ implies } \MAK, (x, y') \models \varphi, \\
        \MAK, (x, y) \models @_a \varphi \ &\iff \ \MAK, (a^V, y) \models \varphi, \\
        \MAK, (x, y) \models @_k \varphi \ &\iff \ \MAK, (x, k^V) \models \varphi.
    \end{align*}
\end{dfn}

The truth of each propositional variable is determined by either $x \in W_A$ or $y \in W_K$. Especially whether $p_K$ is true or false is independent of the element of $W_A$, so $p_K$ can be assumed as an agent-free proposition.

The usage of the satisfaction operator $@$ should also be mentioned. It refers to a specific agent or epistemic alternative while ignoring the current one. For example, the meaning of $@_a \varphi$ is ``for an agent whose name is $a$, $\varphi$ holds.'' The current element of $W_A$ is no longer necessary information to determine the truth of that formula.

\begin{dfn}
    A formula $\varphi$ is \emph{valid} with respect to the class of $\MAK$ (written as $\modelsAK \varphi$) if $\MAK, (x, y) \models \varphi$ for every model $\MAK$ and its every pair $(x, y)$.
\end{dfn}

\subsection{Examples}

As we do in Facebook logic, we can compound friendship and knowledge in agent-knowledge logic. We read $\Box_K \varphi$ as ``I know $\varphi$,'' and $\Box_A \varphi$ as ``All of my friend are $\varphi$.'' For example, we can write some sentences as follows.
\begin{itemize}
    \item $\Box_A \Box_K p_K$: All of my friends know $p_K$.
    \item $\Dia_A \Box_K p_K$: Some of my friends know $p_K$.
    \item $\Box_K \Dia_A \Box_K$: I know that some of my friends know $p_K$.
\end{itemize}

Moreover, we can designate an individual by calling their name owing to nominals. Consider this sentence:
\begin{quotation}
    I am Andy's friend, and Andy knows he has that the Earth goes around the Sun Then, one of my friends knows the heliocentric theory.
\end{quotation}
This inference can be symbolized in the agent-knowledge logic as follows:
\[
    \Dia_A a \land @_a \Box_K p_K \rightarrow \Dia_A \Box_K p_K ,
\]
where $p_K$ shows ``the Earth goes around the Sun'' and $a$ shows ``This is Andy.''

The difference between agent-knowledge logic and Facebook logic becomes more pronounced when we assume that the binary relations over epistemic alternatives are equivalence relations. For example, in Facebook logic, the formula $@_n K p \rightarrow p$ is not valid even if $\asymp_w$ is an equivalence relation. Define $\Kmodel = (W, A, (\sim_a)_{a \in A}, (\asymp_w)_{w \in W}, V)$ as follows:
\begin{align*}
    W &= \{ w, v \}, \\
    A &= \{ a, b \}, \\
    \sim_a &= \sim_b = W \times W, \\
    \asymp_w &= \asymp_v = \emptyset, \\
    V(p) &= \{ (w, b), (v, b)\}, \\
    V(n) &= W \times \{ b \}.
\end{align*}
Then, $\Kmodel, (w, a) \models @_n K p$ holds but we have $\Kmodel, (w, a) \not \models p$. However, in agent-knowledge logic, the situation changes.
\begin{prop}
    The formula $@_a \Box_K p_K \rightarrow p_K$ is valid with respect to the class of $\MAK$ where all of $S_x$ are equivalence relations.
\end{prop}
\begin{proof}
    Suppose that $\MAK, (x, y) \models @_a \Box_K p_K$. Then, we have $\MAK, (a^V, y) \models \Box_K p_K$. By the reflexivity of $S_y$, especially we have $\MAK, (a^V, y) \models p_K$. Since the truth value of $p_K$ is determined only by an element of $W_K$, we have $\MAK, (x, y) \models p_K$.
\end{proof}
This fact may be better understood if we interpret those formulas into natural language. Even though Andy knows he has pollen allergy, it does not mean so does I. However, if he knows that the Earth goes around the Sun, then it is true; the Earth really goes around the Sun.

In addition to the relationships between epistemic alternatives, we can also impose restrictions on the relationships between agents as needed. For example, in Facebook logic, the relationship between agents should be irreflexive and symmetric. Also, we have another way to capture relationships between agents, for example, to read $x R_y x'$ as ``in the situation $y$, the agent $x$ can see the post of $x'$'' in X\footnote{Most of the readers are familiar with the name once it had; twitter.}. Then, we can read $\Box_A \Box_K p_K$ as ``all the people know $p_K$, \emph{as far as I know}.''

\section{Embedding Epistemic Logic into Agent-Knowledge Logic}
\label{secembed}
One of the aims of our new logic is to make it an alternative to Facebook logic. In fact, any sentence we can express in basic epistemic logic can be rewritten in this agent-knowledge logic. In this section, we show that we can embed epistemic logic into agent-knowledge logic.

First of all, we identify the theorem we wish to prove. A proper translation $T$ exists, and the following theorem holds.

\begin{thm}
    For all $\varphi \in \langEL$,
    \[
        \modelsEL \varphi \ \iff \ \modelsAK T(\varphi).
    \]
    \label{thmembed}
\end{thm}

To prove it, let us define how to translate a formula of epistemic logic.

\begin{dfn}
    We define a \emph{translation} $T: \langEL \rightarrow \langAK$ as follows:
    \begin{align*}
        T: \Prop \ni p &\mapsto p_K \in \Prop_K \text{ is a bijection,} \\
        T: \agent \ni i &\mapsto a \in \Nom_A \text{ is a bijection,} \\
        T(\neg \varphi) &= \neg T(\varphi), \\
        T(\varphi \land \psi) &= T(\varphi) \land T(\psi), \\
        T(K_i \varphi) &= @_{T(i)} \Box_K T(\varphi).
    \end{align*}
\end{dfn}

\begin{ex}
    Here is one example of translation.
    \[
        T(K_i (p \land K_j \neg q)) = @_{a_i} \Box_K (p_K \land @_{a_j} \Box_K \neg q_K).
    \]
    We write $a_i$ to abbreviate $T(i) \ (i \in \agent)$.
\end{ex}

In fact, the idea of rewriting $K_i \varphi$ as $@_{T(i)} \Box_K T(\varphi)$ was presented in Sano's review in 2011 \cite{sano2011} for Japanese, which introduces Seligman et al. \cite{seligman2011}. Unfortunately, this translation does not work for Facebook logic, but it does work when the target logic is agent-knowledge logic.

\begin{dfn}
    Given an EL model $\MEL = (W, (R_i)_{i \in \agent}, V)$, the induced AK model $\MAK^\alpha$ is defined as follows:

    $\MAK^\alpha = (\agent, W, \emptyset, (R_i)_{i \in \agent}, V^\alpha)$, where
    \begin{itemize}
        \item For any $p_A \in \Prop_A$, $V^\alpha(p_A) = \emptyset$,
        \item For any $p_K \in \Prop_K$, $V^\alpha(p_K) = V(T^{-1}(p_K))$,
        \item For any $a \in \Nom_A$, $V^\alpha(a) = V(T^{-1}(a))$,
        \item Take one $y_0 \in W$, and for any $k \in \Nom_K$, $V^\alpha(k) = \{ y_0 \}$.
    \end{itemize}
\end{dfn}

Note that we do not care about the definitions of $(R_y)_{y \in W_K}, V^\alpha(p_A),$ and $V^\alpha(k)$. It is because the formula translated by $T$ requires only $\Prop_K, \Nom_A,$ Boolean operators, $\Box_K$, and $@_a$.

\begin{lem}
    For any $\varphi \in \langEL$ and for any $i \in \agent$, we have:
    \[
        \MEL, w \models \varphi \iff \MAK^\alpha, (i, w) \models T(\varphi).
    \]
    \label{lemELtoAK}
\end{lem}
\begin{proof}
    By induction on the complexity of $\varphi$.
    \begin{description}
        \item[$(\varphi = p)$] For all $i \in \agent$,
            \begin{align*}
                \MEL, w \models p &\iff w \in V(p) \\
                &\iff w \in V^\alpha(T(p)) \\
                &\iff \MAK^\alpha, (i, w) \models T(p).
            \end{align*}
        \item[$(\varphi = \neg \psi, \psi \land \chi)$] Straightforward.
        \item[$(\varphi = K_j \psi)$] First, we prove the left-to-right direction.

            Suppose that $\MEL, w \models K_j \psi$. Then, for all $v$ such that $w R_j v$, we have $\MEL, v \models \psi$. We divide the proof into two cases depending on whether such a world $v \in W$ exists.
            \begin{enumerate}[(i) ]
                \item If there is some $v \in W$, take arbitrary one. Then, we have $\MEL, v \models \varphi$. By the induction hypothesis, especially $\MAK^\alpha, (j, v) \models T(\varphi)$. Since we took $v$ arbitrarily, it follows that $\MAK^\alpha, (j, w) \models \Box_K T(\varphi)$. By the definition of $V^\alpha$, we finally get that $\MAK^\alpha, (i, w) \models @_{T(j)}\Box_K T(\varphi)$ for all $i \in A$.
                \item If there is no $v \in W$ such that $w R_j v$, we straightforwardly get that $\MAK^\alpha, (j, w) \models \Box_K T(\varphi)$. In the same way as in the former case, we have $\MAK^\alpha, (i, w) \models @_{T(j)} \Box_K T(\varphi)$ for all $i \in \agent$.
            \end{enumerate}
            In both cases, we can reach the result that $\MAK^\alpha, (i, w) \models @_{T(j)} \Box_K T(\varphi)$ for all $i \in \agent$. Therefore, we have $\MAK^\alpha, (i, w) \models T(K_j \varphi)$.

            Next, we prove the other direction. Take one $i \in A$ and suppose that $\MAK^\alpha, (i, w) \models T(K_j \psi)$. It means that for all $v$ such that $w R_j v$, $\MAK^\alpha, (j, v) \models T(\psi)$ holds. Take one $v$ such that $w R_j v$ (if we cannot, then $\MEL, w \models K_j \psi$ is straightforward.) By the induction hypothesis, we have $\MEL, v \models \psi$. Since we took $v$ arbitrarily, it follows that $\MEL, w \models K_j \psi$.
    \end{description}
\end{proof}

\begin{dfn}
    Given an AK model $\MAK = (W_A, W_K, (R_y)_{y \in W_K}, (S_x)_{x \in W_A}, V)$, the induced EL model $\MEL^\beta$ is defined as follows:

    $\MEL^\beta = (W_K, (S_i^\beta)_{i \in \agent}, V^\beta$), where
    \begin{itemize}
        \item $\agent$ is the set used in Definition \ref{deflangEL},
        \item $y S^\beta_i z$ in $\MEL^\beta$ iff $y S_{T(i)^V} z$ in $\MAK$,
        \item $V^\beta(p) = V(T(p))$.
    \end{itemize}
\end{dfn}

Let us consider a function $\beta_A: W_A \rightarrow \agent$ such that $\beta(T(i)^V) = i$ for all $i \in \agent$. It expresses the correspondence between an agent in $W_A$ and an agent in $\agent$. The illustration of this condition in Figure \ref{figcyclic} may help your understanding.

\begin{figure}[t]
    \centering
        \begin{tikzpicture}
            \node (i) [label=left:$\agent \ni$] at (0,0){$i$};
            \node (a) [label=right:$\in \Nom_1$] at (2,0) {$a$};
            \node (x) [label=right:$\in W_A$] at (2, -2) {$x$};
            \node (ix) [label=left:{$\agent \ni$}] at (0, -2){$i$};
            \draw (i) [auto, |-stealth] -- node {$T$} (a);
            \draw (a) [auto, |-stealth] -- node {$(\cdot)^V$} (x);
            \draw (x) [auto, |-stealth] -- node {$\beta_A$} (ix);
            \draw (ix) [dashed, auto, |-stealth] -- node {id$_A$} (i);
        \end{tikzpicture}
    \caption{The condition $\beta_A$ satisfies (id$_A$ is the identity on $\agent$).}
    \label{figcyclic}
\end{figure}
\begin{lem}
    For any $\varphi \in \langEL$ and for any $x \in W_A$,
    \[
        \MAK, (x, y) \models T(\varphi) \iff \MEL^\beta, y \models \varphi.
    \]
    \label{lemAKtoEL}
\end{lem}
\begin{proof}
    By induction on the complexity of $\varphi$.
    \begin{description}
        \item[$(\varphi = p)$] For all $x \in W^A$,
            \begin{align*}
                \MAK, (x, y) \models T(p) &\iff y \in V(T(p)) \\
                &\iff y \in V^\beta(p) \\
                &\iff \MEL^\beta, y \models p.
            \end{align*}
        \item[$(\varphi = \neg \psi, \psi \land \chi)$] Straightforward.
        \item[$(\varphi = K_j \psi)$] First, we prove the left-to-right direction.

        Suppose that $\MAK, (x, y) \models T(K_j \psi)$. That is, we assume that $\MAK, (x, y) \models @_{T(j)} \Box_K T(\psi)$. Then, for all $z$ such that $y S_{T(j)^V} z$, we have $\MAK, (T(j)^V, z) \models T(\psi)$.
        Bearing the definition of $S_i^\beta$, it suffices to pick up one $z \in W_K$ such that $y S^\beta_j z$ (if we cannot, it is straightforward that $\MEL^\beta, y \models K_j \psi$ holds.) By the assumption, we have $\MAK, (T(j)^V, z) \models T(\psi)$. By the induction hypothesis, $\MEL^\beta, z \models \psi$. Since we picked up $z$ arbitrarily, we have $\MEL^\beta, y \models K_j \psi$.

        Next, we prove the other direction. Suppose that $\MEL^\beta, y \models K_j \psi$. It means that for all $z \in W_K$ such that $y S_j^\beta z$, $\MEL^\beta, z \models \psi$ holds. Now, pick $z \in W_A$ such that $y S_{T(j)^V} z$ arbitrarily (if we cannot, we have $\MAK, (x, y) \models T(K_j \psi)$ for all $x \in W_A$.), and we have $y S_j^\beta z$. Then, we have $\MAK^\beta, z \models \psi$. By the induction hypothesis, $\MAK, (T(j)^V, z) \models T(\psi)$. Since we pick up $z$ arbitrarily, it follows that $\MAK, (x, y) \models @_{T(j)} \Box_K T(\psi)$ for any $x \in W_A$, which means $\MAK, (x, y) \models T(K_j \psi)$.
    \end{description}
\end{proof}

Now, we are ready to prove the main theorem, Theorem \ref{thmembed}. Here is the proof.

\begin{proof}
    We prove it by showing the contraposition. To prove the left-to-right direction, suppose that we have some $\varphi$ such that $\not \modelsAK T(\varphi)$. Then, there is a model $\MAK$ and its pair of points $(x, y)$ such that $\MAK, (x, y) \models \neg T(\varphi)$, which means that $\MAK, (x, y) \models T(\neg \varphi)$. Then, by Lemma \ref{lemAKtoEL}, we have $\MEL^\beta, y \models \neg \varphi$, which leads us to the conclusion that $\not \modelsEL \varphi$. The case of the other direction can be done by using Lemma \ref{lemELtoAK}.
\end{proof}

We usually treat binary relations of EL models as equivalence relations. Moreover, once we want to deal with beliefs by means of a modal operator, we impose yet another condition on accessibility relations. The following corollary shows how embedding can reflect these restrictions.

\begin{prop}
    We have the following properties:
    \begin{enumerate}[(i) ]
        \item For every $i \in \agent$, if $R_i$ in $\MEL$ is reflexive (or serial, symmetric, transitive, euclidian), then so is $R_i$ in $\MAK^\alpha$.
        \item For every $x \in W_A$, if $S_x$ in $\MAK$ is reflexive (or serial, symmetric, transitive, euclidian), then so is $S_i^\beta$ in $\MEL^\beta$.
    \end{enumerate}
\end{prop}
\begin{proof}
    The former is obvious, and the latter is straightforward from the definition of $S_i^\beta$.
\end{proof}

\section{Proof System}
\label{sectableau}

In this section, we introduce a tableau calculus as a proof system.

In constructing a tableau calculus for agent-knowledge logic, we have made significant references to that for hybrid logic. The primary reference is the work of Bolander and Blackburn \cite{bolander2007}. We also refer to Nishimura \cite{nishimura2024}, which studies tableau calculi for some two-dimensional hybrid logics.

For simplicity, this section deals only with the negation normal form (NNF, in short) of formulas. For the satisfaction operators, a formula $\neg @_a \varphi$ is equivalent to $@_a \neg \varphi$. That is, for any model and its possible world $(x, y)$, a formula $\varphi$, and a nominal $a \in \Nom_A$, we have
\[
    \MAK, (x, y) \models @_a \neg \varphi \iff \MAK, (x, y) \models \neg @_a \varphi .
\]
The same equivalence holds for the case of $k \in \Nom_K$. Transformations to the NNF involving Boolean and modal operators can be done in the usual way.

\subsection{Tableau Calculus}

Here we provide a tableau calculus of agent-knowledge logic, denoted by $\tableauAK$.

\begin{dfn}
    A \emph{tableau} is a well-founded tree constructed in the following way:
    \begin{itemize}
      \item Start with a formula of the form $@_a @_k \varphi$ (called the \emph{root formula}), where $\varphi$ is a formula of agent-knowledge logic and $a \in \Nom_A, k \in \Nom_K$ does not occur in $\varphi$.
      \item For each branch, extend it by applying rules (see Definition \ref{defrules}) to all nodes as often as possible. However, we can no longer add any formula in a branch if at least one of the following conditions is satisfied:
      \begin{enumerate}[(i)]
        \item Every new formula generated by applying any rule already exists in the branch.
        \item The branch is closed (see Definition \ref{defclose}.)
      \end{enumerate}
    \end{itemize}
    Here, a \emph{branch} means a maximal path of a tableau. If a formula $\varphi$ occurs in a branch $\Theta$, we write $\varphi \in \Theta$.
    \label{deftableau}
  \end{dfn}

  \begin{dfn}
    A branch of a tableau $\Theta$ is \emph{closed} if one of the following condition holds.
    \begin{enumerate}[(i)]
      \item There are $a \in \Nom_A$, $k, l \in \Nom_K$, and $p_A \in \Prop_A$ such that $@_a @_k p_A, @_a @_l \neg p_A \in \Theta$.
      \item There are $a, b \in \Nom_A$, $k \in \Nom_K$, and $p_K \in \Prop_K$ such that $@_a @_k p_K, @_b @_k \neg p_K \in \Theta$.
      \item There are $a, b \in \Nom_A$ and $k, l \in \Nom_K$ such that $@_a @_k b, @_a @_l \neg b \in \Theta$.
      \item There are $a, b \in \Nom_A$ and $k, l \in \Nom_K$ such that $@_a @_k l, @_b @_k \neg l \in \Theta$.
    \end{enumerate}
    We say that $\Theta$ is \emph{open} if it is not closed. A tableau is called \emph{closed} if all branches in the tableau are closed.
    \label{defclose}
  \end{dfn}

\begin{dfn}
    We provide the rules of $\tableauAK$ in Figure \ref{figrules}.
    \begin{figure}[t]
    \begin{gather*}
        \infer[{[\reflexivity_A]}^{*1}]
            {@_a @_k a}
            { }
        \qquad
        \infer[{[\reflexivity_K]}^{*1}]
            {@_a @_k k}
            { }  \\
        \\
        \infer[{[\neg \neg]}]
            {@_a @_k \varphi}
            {@_a @_k \neg \neg \varphi}
        \qquad
        \infer[{[\land]}]
            {\deduce{@_a @_k \psi} {@_a @_k \varphi}}
            {@_a @_k (\varphi \land \psi)}
        \qquad
        \infer[{[\lor]}]
            {@_a @_k \varphi \mid @_a @_k \psi}
            {@_a @_k (\varphi \lor \psi)} \\
        \\
        \infer[{[\Dia_A]}^{*2, *3, *4}]
            {\deduce{@_b @_k \varphi}{@_a @_k \Dia_A b}}
            {@_a @_k \Dia_A \varphi}
        \qquad
        \infer[{[\Dia_K]}^{*2, *3, *5}]
            {\deduce{@_a @_l \varphi}{@_a @_k \Dia_K l}}
            {@_a @_k \Dia_K \varphi}
        \qquad
        \infer[{[\Box_A]}^{*6}]
            {@_b @_k \varphi}
            {\deduce{@_a @_k \Dia_A b}{@_a @_k \Box_A \varphi}}
        \qquad
        \infer[{[\Box_K]}^{*6}]
            {@_a @_l \varphi}
            {\deduce{@_a @_k \Dia_K l}{@_a @_k \Box_K \varphi}} \\
        \\
        \infer[{[@_A]}]
            {@_b @_k \varphi}
            {@_a @_k @_b \varphi}
        \qquad
        \infer[{[@_K]}]
            {@_a @_l \varphi}
            {@_a @_k @_l \varphi}
        \qquad
        \infer[{[\Id_A]}^{*3}]
            {@_b @_k \varphi}
            {\deduce{@_a @_k b} {@_a @_k \varphi}}
        \qquad
        \infer[{[\Id_K]}^{*3}]
            {@_a @_l \varphi}
            {\deduce{@_a @_k l} {@_a @_k \varphi}}
    \end{gather*}

    *1: $a \in \Nom_A$ and $k \in \Nom_K$ have already occurred in the branch.

    *2: This rule can be applied only one time per formula.

    *3: The formula above the line is not an accessibility formula. Here, an \emph{accessibility formula} is the formula of the form $@_a @_k \Dia_A b$ ($@_a @_k \Dia_K l$) generated by $[\Dia_A]$ ($[\Dia_K]$), where $b$ ($l$) is a new nominal.

    *4: $b \in \Nom_A$ does not occur in the branch.

    *5: $l \in \Nom_K$ does not occur in the branch.

    *6: The second formula above the line is an accessibility formula.

    \vspace{\baselineskip}
    In these rules, the formulas above the line show the formulas that have already occurred in the branch, and the formulas below the line show the formulas that will be added to the branch. The vertical line in the $[\lor]$ means that the branch splits to the left and right.
    \caption{The rules of $\tableauAK$}
    \label{figrules}
    \end{figure}
    \label{defrules}
\end{dfn}

\begin{dfn}[provability]
    Given a formula $\varphi$, we say that $\varphi$ is \emph{provable} in $\tableauAK$ if there is a closed tableau whose root formula is $@_a @_k \varphi'$, where $a \in \Nom_A$ and $k \in \Nom_K$ does not occur in $\varphi$, and $\varphi'$ is an NNF of $\neg \varphi$.
\end{dfn}

\subsection{Termination and Completeness}

A tableau calculus has \emph{the termination property} if, for any tableau constructed in the system, all branches have finite length. We firstly prove that the tableau calculus $\tableauAK$ introduced above has the termination property. Due to the limited space of the paper, we provide a brief outline of the proof.

\begin{dfn}
    Let $\Theta$ be a branch of a tableau, and let $a, b \in \Nom_A$ and $k, l \in \Nom_K$ be nominals occurring in $\Theta$. A pair $(b, l)$ of nominals is \emph{generated} by $(a, k)$ in $\Theta$ (written: $(a, k) \prec_\Theta (b, l)$) if one of the following conditions holds.
    \begin{enumerate}[(i)]
        \item $k = l$ and $b$ is introduced by applying $[\Dia_A]$ to $@_a @_k \Dia_A \varphi$.
        \item $a = b$ and $l$ is introduced by applying $[\Dia_K]$ to $@_a @_k \Dia_K \varphi$.
    \end{enumerate}
\end{dfn}

\begin{lem}
    Let $\Theta$ be a branch of a tableau. The length of $\Theta$ is infinite if and only if there is an infinite sequence
    \[
        (a_0, k_0) \prec_\Theta (a_1, k_1) \prec_\Theta \cdots .
    \]
    \label{leminf}
\end{lem}

\begin{dfn}
    Let $\Theta$ be a branch of a tableau, and let $a \in \Nom_A$ and $k \in \Nom_K$ be nominals occurring in $\Theta$. We define a function $m_\Theta: \Nom_A \times \Nom_K \rightarrow \mathbb{N}$ as follows:
    \[
        m_\Theta((a, k)) = \max\{ |\varphi| \mid @_a @_k \varphi \in \Theta \} .
    \]
\end{dfn}

\begin{lem}
    Let $\Theta$ be a branch of a tableau. If $(a, k) \prec_\Theta (b, l)$, then $m_\Theta((a, k)) > m_\Theta((b, l))$.
    \label{lemdec}
\end{lem}

\begin{thm}
    The tableau calculus $\tableauAK$ has the termination property.
\end{thm}
\begin{proof}
    By \emph{reductio ad absurdum}. Suppose there is a branch $\Theta$ of a tableau that is infinite. Then, by Lemma \ref{leminf}, we have an infinite sequence
    \[
        (a_0, k_0) \prec_\Theta (a_1, k_1) \prec_\Theta \cdots .
    \]
    Applying Lemma \ref{lemdec}, we have an infinite decreasing sequence
    \[
        m_\Theta((a_0, k_0)) > m_\Theta((a_1, k_1)) > \cdots ,
    \]
    which contradict to the definition of $m_\Theta$.
\end{proof}

The soundness of $\tableauAK$ can be proved in a similar way introduced in \cite{nishimura2024}. Then, we move on to prove the completeness of $\tableauAK$. In preparation, we define some terms.

First, we use the term \emph{subformula} with an expanded meaning. Given two formulas $@_a @_k \varphi$ and $@_b @_l \psi$, the formula $@_a @_k \varphi$ is a \emph{subformula} of the other formula $@_b @_l \psi$ if $\varphi$ is a subformula (in the usual way) of $\psi$. Second, we say a branch $\Theta$ \emph{saturated} if every new formula generated by applying some rules already exists in $\Theta$.

\begin{dfn}
    Given a branch $\Theta$ of a tableau, we define $\sim^A_\Theta \subset \Nom_A \times \Nom_A$ and $\sim^K_\Theta \subset \Nom_K \times \Nom_K$ as follows.
    \begin{itemize}
        \item $a \sim^A_\Theta b$ if there is a nominal $k \in \Nom_K$ such that $@_a @_k b \in \Theta$.
        \item $k \sim^K_\Theta l$ if there is a nominal $a \in \Nom_A$ such that $@_a @_k l \in \Theta$.
    \end{itemize}
\end{dfn}

We can show that if $\Theta$ is saturated, then both $\sim^A_\Theta$ and $\sim^K_\Theta$ are equivalence relations. They enable us to take a representative of nominals.

\begin{dfn}
    Let $\Theta$ be a tableau branch and $a \in \Nom_A$ a nominal occurring in $\Theta$. The \emph{urfather} of $a$ on $\Theta$ (written: $u_\Theta(a)$) is the earliest introduced nominal $b$ such that $a \sim^A_\Theta b$. For $k \in \Nom_K$, $u_\Theta(k)$ is defined in the same way.
\end{dfn}

\begin{dfn}
    Given an open saturated branch $\Theta$, a model $\MAK^\Theta = (W_A^\Theta, W_K^\Theta, (R_y^\Theta)_{y \in W_K^\Theta}, (S_x^\Theta)_{x \in W_A^\Theta}, V^\Theta)$ generated from $\Theta$ is defined as follows:
    \begin{align*}
        W_A^\Theta &= \{ u_\Theta(a) \mid a \in \Nom_A \text{ occurs in } \Theta \}, \\
        W_K^\Theta &= \{ u_\Theta(k) \mid k \in \Nom_K \text{ occurs in } \Theta \}, \\
        R_{u_\Theta(k)}^\Theta &= \{ (u_\Theta(a), u_\Theta(b)) \mid \text{accessibility formula } @_a @_k \Dia_A b \in \Theta \}, \\
        S_{u_\Theta(a)}^\Theta &= \{ (u_\Theta(k), u_\Theta(l)) \mid \text{accessibility formula } @_a @_k \Dia_K l \in \Theta \}, \\
        V^\Theta(p_A) &= \{ u_\Theta(a) \mid \text {there is } k \in \Nom_K \text{ such that } @_a @_k p_A \in \Theta \} \text { where } p_A \in \Prop_A, \\
        V^\Theta(p_K) &= \{ u_\Theta(k) \mid \text {there is } a \in \Nom_A \text{ such that } @_a @_k p_K \in \Theta \} \text{ where } p_K \in \Prop_K, \\
        V^\Theta(a) &= \{ u_\Theta(a) \} \text{ where } a \in \Nom_A, \\
        V^\Theta(k) &= \{ u_\Theta(k) \} \text{ where } k \in \Nom_K.
    \end{align*}
\end{dfn}

\begin{lem}
    Let $\Theta$ be an open saturated branch and let $@_a @_k \varphi$ be a subformula of the root formula of $\Theta$. Then, we have:
    \[
        \text{if } @_a @_k \varphi \in \Theta \text{, then } \MAK^\Theta, (u_\Theta(a), u_\Theta(k)) \models \varphi .
    \]
    \label{lemmodelexist}
\end{lem}

This lemma is called \emph{model existence lemma}. Note that by combining it with the termination property of $\tableauAK$, we can show the finite model property of agent-knowledge logic as well as the completeness.

\begin{thm}
    The tableau calculus $\tableauAK$ is complete for the class of all AK models.
\end{thm}
\begin{proof}
    We show the contraposition.

    Suppose that $\varphi$ is not provable in $\tableauAK$. Then, we can find an open and saturated branch $\Theta$ with the root formula $@_a @_k \varphi'$, where $a \in \Nom_A$ and $k \in \Nom_K$ does not occur in $\varphi$, and $\varphi'$ is
    an NNF of $\neg \varphi$. Then, by Lemma \ref{lemmodelexist}, we have $\MAK^\Theta, (u_\Theta(a), u_\Theta(k)) \models \varphi'$. It means that there is an AK model and its possible world which falsify $\varphi$.
\end{proof}

The termination property and completeness of the tableau calculus tell us about the decidability of logic. If $\varphi$ is provable, then it is provable in finite time. By contrast, if $\varphi$ is unprovable, we can make a finite counterexample model. From them, the following corollary holds.

\begin{cor}
    The agent-knwoledge logic is decidable.
\end{cor}

\section{Future Work and Perspective}
\label{secfuture}

\subsection{Seeking More Usage}
As one of the expected future research endeavours involving agent-knowledge logic I plan to examine a greater variety of representations. The use of the tools presented in this paper would be just the tip of the iceberg. For example, $\Nom_A$ is the set of agents, and $\Prop_K$ is an agent-independent proposition. However, it is difficult to say that sufficient utilization has been found for $\Prop_A$ and $\Nom_K$.

It is also fruitful to imitate various operators of epistemic logic. For example, given a group $G \subseteq \agent$ of agents, the \emph{everybody knows operator} $E_G$ is defined as follows:
\[
    \MEL, w \models E_G \varphi \ \iff \ \Kmodel, w \models K_i \varphi \text{ for all } i \in G .
\]
Intuitively, this formula says that everyone in the group $G$ knows $\varphi$. In agent-knowledge logic, $E_G \varphi$ can be expressed by the following formula:
\[
    \bigwedge_{i \in G} @_{T(i)} \Box_K T(\varphi).
\]
Also, we may mimic other operators used in epistemic logic, such as the operator for common knowledge $C_G$ and the operator for distributed knowledge $D_G$. Research in this direction may be able to reflect various results in epistemic logic in agent-knowledge logic as well.

Additionally, there is another direction to research about agent-knowledge logic, to introduce a universal operator used in hybrid logic, which may enable us to symbolize more expressions in natural language. The definition of the universal operators $\mathsf{A}_A$ and $\mathsf{E}_A$ are as follows:
\begin{align*}
    \MAK, (x, y) \models \mathsf{A}_A \varphi \ &\iff \ \MAK, (z, y) \models \varphi \text{ for all } z \in W_A, \\
    \MAK, (x, y) \models \mathsf{E}_A \varphi \ &\iff \ \text{there is some } z \in W_A \text{ such that } \MAK, (z, y) \models \varphi.
\end{align*}

Owing to these operators, we can write some expressions as follows:
\begin{itemize}
    \item $\mathsf{E}_A \Box_K p_K$: Someone knows $p_K$.
    \item $\Box_K \mathsf{A}_A \Box_K p_K$: I know that all the people know $p_K$.
    \item $\mathsf{E}_A \Box_A \Box_K p_K$: There is a person all of whose friends know $p_K$.
\end{itemize}

\subsection{Hilbert-Style Axiomatization}
In this paper, we have given the tableau calculus for agent-knowledge logic as a proof system. We can give another proof system, for example, Hilbert-style axiomatization.

Fortunately, there is already abundant previous research in the surrounding fields. In addition to the aforementioned Sano's work \cite{sano2010}, Balbiani and Fern\'{a}ndez Gonz\'{a}lez \cite{balbiani2020} has shown the Hilbert-style axiomatization of Facebook logic. For LHS, recent research by Chen and Li \cite{chen2024} gives the axiomatization.

\subsection{Complexity}
In this paper, we have shown the decidability of the agent-knowledge logic using tableau calculus. But what about its computational complexity? As already known, the satisfiability problem for epistemic logic is PSPACE-complete \cite{halpern1992}. If agent-knowledge logic is used as an alternative to epistemic logic, it must also be PSPACE-complete.

The analysis of computational complexity for a fusion in modal logic may provide a clue to solving this problem. An explanation for a fusion is in \cite[p. 111]{gabbay2003}:
\begin{quotation}
    Let $L_1$ and $L_2$ be two multimodal logics formulated in languages $\mathcal{L}_1$ and $\mathcal{L}_2$, both containing the language $\mathcal{L}$ of classical propositional logic, but having disjoint sets of modal operators. Denote by $\mathcal{L}_1 \otimes \mathcal{L}_2$ the union of $\mathcal{L}_1$ and $\mathcal{L}_2$. Then the \emph{fusion} $L_1 \otimes L_2$ of $L_1$ and $L_2$ is the smallest multimodal logic $L$ in the language $\mathcal{L}_1 \otimes \mathcal{L}_2$ containing $L_1 \cup L_2$.
\end{quotation}
From the results of Halpern and Moses \cite{halpern1992}, we can obtain that the satisfiability problem for $\mathbf{K} \otimes \mathbf{K}$ is PSPACE-complete. Since agent-knowledge logic is based on $\mathbf{K} \otimes \mathbf{K}$, we may be able to answer the question with reference to this proof.

\section*{Acknowledgements}

I would like to thank Prof. Ryo Kashima and Leonardo Pacheco for their invaluable advice in writing this paper. The research of the author was supported by JST SPRING, Grant Number JPMJSP2106.

\bibliographystyle{eptcs}
\bibliography{logic}

\begin{thebibliography}{10}
\providecommand{\bibitemdeclare}[2]{}
\providecommand{\surnamestart}{}
\providecommand{\surnameend}{}
\providecommand{\urlprefix}{Available at }
\providecommand{\url}[1]{\texttt{#1}}
\providecommand{\href}[2]{\texttt{#2}}
\providecommand{\urlalt}[2]{\href{#1}{#2}}
\providecommand{\doi}[1]{doi:\urlalt{https://doi.org/#1}{#1}}
\providecommand{\eprint}[1]{arXiv:\urlalt{https://arxiv.org/abs/#1}{#1}}
\providecommand{\bibinfo}[2]{#2}

\bibitemdeclare{inproceedings}{balbiani2020}
\bibitem{balbiani2020}
\bibinfo{author}{Philippe \surnamestart Balbiani\surnameend} \&
  \bibinfo{author}{Sa{\'u}l~Fern{\'a}ndez \surnamestart
  Gonz{\'a}lez\surnameend} (\bibinfo{year}{2020}):
  \emph{\bibinfo{title}{Indexed Frames and Hybrid Logics}}.
\newblock In: {\slshape \bibinfo{booktitle}{Advances in Modal Logic}}.

\bibitemdeclare{article}{vanbenthem2006}
\bibitem{vanbenthem2006}
\bibinfo{author}{Johan \surnamestart van Benthem\surnameend}
  (\bibinfo{year}{2006}): \emph{\bibinfo{title}{Epistemic Logic and
  Epistemology: The State of their Affairs}}.
\newblock {\slshape \bibinfo{journal}{Philosophical Studies}}
  \bibinfo{volume}{128}, pp. \bibinfo{pages}{249--76},
  \doi{10.1007/s11098-005-4052-0}.

\bibitemdeclare{article}{blackburn2006P}
\bibitem{blackburn2006P}
\bibinfo{author}{Patrick \surnamestart Blackburn\surnameend} \&
  \bibinfo{author}{Balder \surnamestart ten Cate\surnameend}
  (\bibinfo{year}{2006}): \emph{\bibinfo{title}{Pure Extensions, Proof Rules,
  and Hybrid Axiomatics}}.
\newblock {\slshape \bibinfo{journal}{Studia Logica}}
  \bibinfo{volume}{84}(\bibinfo{number}{2}), pp. \bibinfo{pages}{277--322},
  \doi{10.1007/s11225-006-9009-6}.

\bibitemdeclare{article}{bolander2007}
\bibitem{bolander2007}
\bibinfo{author}{Thomas \surnamestart Bolander\surnameend} \&
  \bibinfo{author}{Patrick \surnamestart Blackburn\surnameend}
  (\bibinfo{year}{2007}): \emph{\bibinfo{title}{Termination for Hybrid
  Tableaus}}.
\newblock {\slshape \bibinfo{journal}{Journal of Logic and Computation}}
  \bibinfo{volume}{17}(\bibinfo{number}{3}), pp. \bibinfo{pages}{517--554},
  \doi{10.1093/logcom/exm014}.

\bibitemdeclare{book}{brauner2011}
\bibitem{brauner2011}
\bibinfo{author}{Torben \surnamestart Bra{\"u}ner\surnameend}
  (\bibinfo{year}{2011}): \emph{\bibinfo{title}{Hybrid Logic and its
  Proof-Theory}}.
\newblock \bibinfo{volume}{37}, \bibinfo{publisher}{Springer Science \&
  Business Media}, \doi{10.1007/978-94-007-0002-4}.

\bibitemdeclare{inproceedings}{chen2024}
\bibitem{chen2024}
\bibinfo{author}{Qian \surnamestart Chen\surnameend} \& \bibinfo{author}{Dazhu
  \surnamestart Li\surnameend} (\bibinfo{year}{2024}):
  \emph{\bibinfo{title}{Logic of the Hide and Seek Game: Characterization,
  Axiomatization, Decidability}}.
\newblock In \bibinfo{editor}{Nina \surnamestart Gierasimczuk\surnameend} \&
  \bibinfo{editor}{Fernando~R. \surnamestart Vel{\'a}zquez-Quesada\surnameend},
  editors: {\slshape \bibinfo{booktitle}{Dynamic Logic. New Trends and
  Applications}}, \bibinfo{publisher}{Springer Nature Switzerland}, pp.
  \bibinfo{pages}{20--34}, \doi{10.1007/978-3-031-51777-8_2}.

\bibitemdeclare{book}{fagin1995}
\bibitem{fagin1995}
\bibinfo{author}{Ronald \surnamestart Fagin\surnameend},
  \bibinfo{author}{Joseph~Y. \surnamestart Halpern\surnameend},
  \bibinfo{author}{Yoram \surnamestart Moses\surnameend} \&
  \bibinfo{author}{Moshe~Y. \surnamestart Vardi\surnameend}
  (\bibinfo{year}{1995}): \emph{\bibinfo{title}{Reasoning about knowledge}}.
\newblock \bibinfo{publisher}{MIT press}, \doi{10.7551/mitpress/5803.001.0001}.

\bibitemdeclare{book}{gabbay2003}
\bibitem{gabbay2003}
\bibinfo{author}{Dov~M. \surnamestart Gabbay\surnameend}
  (\bibinfo{year}{2003}): \emph{\bibinfo{title}{Many-Dimensional Modal Logics:
  Theory and Applications}}.
\newblock \bibinfo{publisher}{Elsevier North Holland}.

\bibitemdeclare{article}{halpern1992}
\bibitem{halpern1992}
\bibinfo{author}{Joseph~Y. \surnamestart Halpern\surnameend} \&
  \bibinfo{author}{Yoram \surnamestart Moses\surnameend}
  (\bibinfo{year}{1992}): \emph{\bibinfo{title}{A Guide to Completeness and
  Complexity for Modal Logics of Knowledge and Belief}}.
\newblock {\slshape \bibinfo{journal}{Artificial Intelligence}}
  \bibinfo{volume}{54}(\bibinfo{number}{3}), pp. \bibinfo{pages}{319--379},
  \doi{10.1016/0004-3702(92)90049-4}.

\bibitemdeclare{book}{Hintikka1962}
\bibitem{Hintikka1962}
\bibinfo{author}{Jaakko \surnamestart Hintikka\surnameend}
  (\bibinfo{year}{1962}): \emph{\bibinfo{title}{Knowledge and Belief: An
  Introduction to the Logic of the Two Notions}}, \bibinfo{edition}{second}
  edition.
\newblock \bibinfo{publisher}{Cornell University Press}.

\bibitemdeclare{article}{indrzejczak2007}
\bibitem{indrzejczak2007}
\bibinfo{author}{Andrzej \surnamestart Indrzejczak\surnameend}
  (\bibinfo{year}{2007}): \emph{\bibinfo{title}{Modal Hybrid Logic}}.
\newblock {\slshape \bibinfo{journal}{Logic and Logical Philosophy}}
  \bibinfo{volume}{16}(\bibinfo{number}{2-3}), pp. \bibinfo{pages}{147--257},
  \doi{10.12775/LLP.2007.006}.

\bibitemdeclare{inproceedings}{li2021}
\bibitem{li2021}
\bibinfo{author}{Dazhu \surnamestart Li\surnameend}, \bibinfo{author}{Sujata
  \surnamestart Ghosh\surnameend}, \bibinfo{author}{Fenrong \surnamestart
  Liu\surnameend} \& \bibinfo{author}{Yaxin \surnamestart Tu\surnameend}
  (\bibinfo{year}{2021}): \emph{\bibinfo{title}{On The Subtle Nature of a
  Simple Logic of The Hide and Seek Game}}.
\newblock In: {\slshape \bibinfo{booktitle}{Logic, Language, Information, and
  Computation: 27th International Workshop, WoLLIC 2021, Virtual Event, October
  5--8, 2021, Proceedings 27}}, \bibinfo{organization}{Springer}, pp.
  \bibinfo{pages}{201--218}, \doi{10.1007/978-3-030-88853-4_13}.

\bibitemdeclare{article}{li2023}
\bibitem{li2023}
\bibinfo{author}{Dazhu \surnamestart Li\surnameend}, \bibinfo{author}{Sujata
  \surnamestart Ghosh\surnameend}, \bibinfo{author}{Fenrong \surnamestart
  Liu\surnameend} \& \bibinfo{author}{Yaxin \surnamestart Tu\surnameend}
  (\bibinfo{year}{2023}): \emph{\bibinfo{title}{A Simple Logic of The Hide and
  Seek Game}}.
\newblock {\slshape \bibinfo{journal}{Studia Logica}}
  \bibinfo{volume}{111}(\bibinfo{number}{5}), pp. \bibinfo{pages}{821--853},
  \doi{10.1007/s11225-023-10039-4}.

\bibitemdeclare{article}{nishimura2024}
\bibitem{nishimura2024}
\bibinfo{author}{Yuki \surnamestart Nishimura\surnameend}
  (\bibinfo{year}{2024}): \emph{\bibinfo{title}{Completeness of Tableau Calculi
  for Two-dimensional Hybrid Logics}}.
\newblock {\slshape \bibinfo{journal}{Journal of Logic and Computation}}, p.
  \bibinfo{pages}{exae018}, \doi{10.1093/logcom/exae018}.

\bibitemdeclare{article}{sano2010}
\bibitem{sano2010}
\bibinfo{author}{Katsuhiko \surnamestart Sano\surnameend}
  (\bibinfo{year}{2010}): \emph{\bibinfo{title}{Axiomatizing Hybrid Products:
  How Can We Reason Many-Dimensionally in Hybrid Logic?}}
\newblock {\slshape \bibinfo{journal}{Journal of Applied Logic}}
  \bibinfo{volume}{8}(\bibinfo{number}{4}), pp. \bibinfo{pages}{459--474},
  \doi{10.1016/j.jal.2010.08.006}.

\bibitemdeclare{article}{sano2011}
\bibitem{sano2011}
\bibinfo{author}{Katsuhiko \surnamestart Sano\surnameend}
  (\bibinfo{year}{2011}): \emph{\bibinfo{title}{Paper Review: Logic in The
  Community}}.
\newblock {\slshape \bibinfo{journal}{Journals of The Japanese Society for
  Artificial Intelligence}} \bibinfo{volume}{26}(\bibinfo{number}{6}), pp.
  \bibinfo{pages}{703--707}, \doi{10.11517/jjsai.26.6_703}.
\newblock \bibinfo{note}{(written in Japanese)}.

\bibitemdeclare{inproceedings}{seligman2011}
\bibitem{seligman2011}
\bibinfo{author}{Jeremy \surnamestart Seligman\surnameend},
  \bibinfo{author}{Fenrong \surnamestart Liu\surnameend} \&
  \bibinfo{author}{Patrick \surnamestart Girard\surnameend}
  (\bibinfo{year}{2011}): \emph{\bibinfo{title}{Logic in the Community}}.
\newblock In: {\slshape \bibinfo{booktitle}{Proceedings of the 4th Indian
  Conference on Logic and Its Applications}}, \bibinfo{publisher}{Springer
  Berlin Heidelberg}, pp. \bibinfo{pages}{178--188},
  \doi{10.1007/978-3-642-18026-2_15}.

\bibitemdeclare{inproceedings}{seligman2013}
\bibitem{seligman2013}
\bibinfo{author}{Jeremy \surnamestart Seligman\surnameend},
  \bibinfo{author}{Fenrong \surnamestart Liu\surnameend} \&
  \bibinfo{author}{Patrick \surnamestart Girard\surnameend}
  (\bibinfo{year}{2013}): \emph{\bibinfo{title}{{F}acebook and the Epistemic
  Logic of Friendship}}.
\newblock In: {\slshape \bibinfo{booktitle}{Proceedings of the 14th Conference
  on Theoretical Aspects of Rationality and Knowledge}}, pp.
  \bibinfo{pages}{229--238}.

\end{thebibliography}

\end{document}